\documentclass[11pt]{amsart}

\usepackage{amscd,amssymb}

\newtheorem{maintheorem}{Theorem}
\newtheorem{maincorollary}{Corollary}
\newtheorem{theorem}{Theorem}[section]
\newtheorem{proposition}[theorem]{Proposition}
\newtheorem{lemma}[theorem]{Lemma}
\newtheorem{corollary}[theorem]{Corollary}
\newtheorem*{firstcorollary}{Corollary~\ref{coro:inf_pi1}}
\newtheorem*{secondcorollary}{Corollary~\ref{coro:highdistance}}
\newtheorem*{generalgenus2theorem}{Theorem~\ref{thm:genus2case}}
\newtheorem*{S1S2theorem}{Theorem~\ref{thm:S1timesS2}}
\newtheorem*{genus0theorem}{Theorem~\ref{thm:S3} part (1)}
\newtheorem*{genus1theorem}{Theorem~\ref{thm:S3} part (2) and Theorem~\ref{thm:lens_spaces} parts (1a), (2a), and (3a)}
\newtheorem*{genus2theorem}{Theorem~\ref{thm:S3} part (3) and Theorem~\ref{thm:lens_spaces} parts (1b), (2b), and (3b)}
\newtheorem*{threefibertheorem}{Theorem~\ref{thm:elliptics}}

\newcommand{\longpage}{\enlargethispage{\baselineskip}}

\newcommand{\Diff}{\operatorname{Diff}}
\newcommand{\diff}{\operatorname{diff}}
\newcommand{\Exp}{\operatorname{Exp}}
\newcommand{\genus}{\operatorname{genus}}
\newcommand{\Imb}{\operatorname{Imb}}
\newcommand{\Img}{\operatorname{Img}}
\newcommand{\isom}{\operatorname{isom}}
\newcommand{\Isom}{\operatorname{Isom}}
\newcommand{\Maps}{C^\infty}
\newcommand{\Mod}{\operatorname{Mod}}
\newcommand{\Norm}{\operatorname{Norm}}
\newcommand{\proj}{\operatorname{proj}}
\newcommand{\Sect}{\operatorname{Sect}}
\newcommand{\TExp}{\operatorname{TExp}}
\newcommand{\rel}{\,\text{rel}\;}

\renewcommand{\O}{\operatorname{O}}
\newcommand{\SO}{\operatorname{SO}}
\newcommand{\C}{\operatorname{{\mathbb C}}}
\newcommand{\R}{\operatorname{{\mathbb R}}}
\newcommand{\RP}{\operatorname{{\mathbb{RP}}}}
\newcommand{\Z}{\operatorname{{\mathbb Z}}}

\renewcommand{\H}{\operatorname{{\mathcal H}}}

\newcommand{\ttimes}{\operatorname{\,\widetilde{\times}\,}}
\newcommand{\Ostar}{\O(2)^*}
\newcommand{\Dih}{\operatorname{Dih}}

\begin{document}

\title[The space of Heegaard splittings]{The space of Heegaard splittings}

\author{Jesse Johnson}
\address{Department of Mathematics\\
Oklahoma State University\\
Stillwater, Oklahoma 74078\\
USA} 
\email{jjohnson@math.okstate.edu}
\urladdr{www.math.okstate.edu/$_{\widetilde{\phantom{n}}}$jjohnson/}
\thanks{The first author was supported in part by NSF MSPRF 0606368}

\author{Darryl McCullough}
\address{Department of Mathematics\\
University of Oklahoma\\
Norman, Oklahoma 73019\\
USA} 
\email{dmccullough@math.ou.edu}
\urladdr{www.math.ou.edu/$_{\widetilde{\phantom{n}}}$dmccullough/}
\thanks{The second author was supported in part by NSF grant DMS-08082424}

\date{\today}


\begin{abstract}For a Heegaard surface $\Sigma$ in a closed orientable 
$3$-manifold $M$, $\H(M,\Sigma)=\Diff(M)/\Diff(M,\Sigma)$ is the space of
Heegaard surfaces equivalent to the Heegaard splitting $(M,\Sigma)$. Its
path components are the isotopy classes of Heegaard splittings equivalent
to $(M,\Sigma)$. We describe $H(M,\Sigma)$ in terms of $\Diff(M)$ and the
Goeritz group of $(M,\Sigma)$. In particular, for hyperbolic $M$ each path
component is a classifying space for the Goeritz group, and when the
(Hempel) distance of $(M,\Sigma)$ is greater than $3$, each path component
of $\H(M,\Sigma)$ is contractible. For splittings of genus $0$ or $1$, we
determine the complete homotopy type (modulo the Smale Conjecture for $M$
in the cases when it is not known).
\end{abstract}

\maketitle

Let $M$ be a closed, orientable $3$-manifold, not necessarily irreducible,
and suppose that $\Sigma$ is a Heegaard surface in $M$. The space
$\H(M,\Sigma)$ of Heegaard splittings equivalent to $(M,\Sigma)$ is defined
to be the space of left cosets $\Diff(M)/\Diff(M,\Sigma)$, where
$\Diff(M,\Sigma)$ is the subgroup of $\Diff(M)$ consisting of
diffeomorphisms taking $\Sigma$ onto $\Sigma$. In other words, this is the
space of images of $\Sigma$ under diffeomorphisms of~$M$.

We will denote the homotopy groups $\pi_i(\H(M,\Sigma))$ by
$\H_i(M,\Sigma)$. In particular, $\H_0(M,\Sigma)$ is the set of isotopy
classes of Heegaard splittings equivalent to $(M,\Sigma)$. In the present
work, we focus on the groups $\H_i(M,\Sigma)$ for~$i\geq 1$. (Note that
$\H_i(M,\Sigma)$ is independent of the basepoint chosen, because $\Diff(M)$
acts transitively on $\H(M,\Sigma)$ and consequently any two path
components are homeomorphic. We use the identity map $1_M$, or more
strictly speaking, the coset $1_M\Diff(M,\Sigma)$, as our implicit choice
of basepoint of $\H(M,\Sigma)$.)

As one would expect, $\H(M,\Sigma)$ is closely related to $\Diff(M)$. When
the genus of $\Sigma$ is at least $2$, the connected components of
$\Diff(\Sigma)$ are contractible, leading to our first main result.
\begin{maintheorem} Suppose that $\Sigma$ has genus at least 
$2$. Then $\pi_q(\Diff(M))\to \H_q(M,\Sigma)$ is an isomorphism for 
$q\geq 2$, and there are exact sequences
\begin{gather*}
1\to \pi_1(\Diff(M))\to \H_1(M,\Sigma)\to G(M,\Sigma) \to 1\ ,\\
1\to G(M,\Sigma)\to \Mod(M,\Sigma)\to \Mod(M)\to \H_0(M, \Sigma)\to 1\ .
\end{gather*}\par
\label{thm:genus2case}
\end{maintheorem}
\noindent In Theorem~\ref{thm:genus2case}, $\Mod(M)$ and $\Mod(M,\Sigma)$
denote the groups of path components of $\Diff(M)$ and $\Diff(M,\Sigma)$
respectively, and $G(M,\Sigma)$ is the \textit{Goeritz group} of the
Heegaard splitting, defined to be the kernel of $\Mod(M,\Sigma)\to
\Mod(M)$. We remark that for most reducible $M$, $\pi_1(\Diff(M))$ is known
to be non-finitely-generated~\cite{K-M}, suggesting that $\H(M,\Sigma)$ has
a complicated homotopy type in these cases.

When $\pi_1(M)$ is infinite, Theorem~\ref{thm:genus2case} applies to all
cases except the genus-$1$ Heegaard surface in $S^1\times S^2$. To state
our result for that case, denote by $LX$ the space of smooth free loops in
a smooth manifold $X$, that is, the $C^\infty$ maps from $S^1$ to $X$, with
the $C^\infty$ topology. There is a free involution $\alpha\colon LS^2\to
LS^2$ defined by $\alpha(\gamma) = \rho \circ \gamma$, where $\rho\colon
S^2\to S^2$ is the antipodal map. The quotient $LS^2/\langle \alpha\rangle$
can be identified with the connected component of the constant loop
in~$L\RP^2$.
\begin{maintheorem} For the unique genus-$1$ Heegaard surface $\Sigma$
in $S^1\times S^2$, $\H(S^1\times S^2,\Sigma)$ is homotopy equivalent to
$LS^2/\langle \alpha \rangle$.
\label{thm:S1timesS2}
\end{maintheorem}

We remark that (at least when $X$ has empty boundary) the inclusion
function from $LX$ to the space of all continuous free loops in $X$ (with
the compact-open topology) is a homotopy equivalence (see
A. Stacey~\cite[Theorem 4.6]{Stacey}). The analogous statement holds for
the space $\Omega X$ of smooth based free loops~\cite[Section
4.3]{Stacey}). The map $LX\to X$ given by evaluation at the basepoint is
a locally trivial fibration~\cite[Corollary 4.8]{Stacey}, with fiber
$\Omega X$. Since $LX\to X$ has an obvious section, the exact sequence of
this fibration shows that $\pi_q(LX)\cong \pi_{q+1}(X)\oplus \pi_q(X)$ for
all $q\geq 1$. The homology of $LS^n$ was computed in
W. Ziller~\cite[p.~21]{Ziller} (see also R. Cohen, J. Jones, and
J. Yan~\cite{Cohen-Jones-Yan}); for $n=2$ it is $H_0(LS^2)\cong \Z$,
$H_k(LS^2)\cong \Z$ for $k>0$ odd, and $H_k(LS^2)\cong \Z\oplus \Z/2$ for
$k>0$ even.

When $\pi_1(M)$ is infinite and $M$ is irreducible, all Heegaard splittings
of $M$ have genus at least~2. In addition, apart from one case in which
$\Diff(M)$ has not been fully determined, we know that $\Diff(M)$ has a
very simple homotopy type. Theorem~\ref{thm:genus2case} becomes the
following statement:
\begin{maincorollary} Suppose that $M$ is irreducible and
$\pi_1(M)$ is infinite, and that $M$ is not a non-Haken infranilmanifold.
Then $\H_i(M,\Sigma)=0$ for $i\geq 2$, and there is an exact sequence
\[ 1\to Z(\pi_1(M)) \to \H_1(M,\Sigma) \to G(M, \Sigma)\to 1\ .\]\par
\label{coro:inf_pi1}
\end{maincorollary}
\noindent Note that when the conclusion of Corollary~\ref{coro:inf_pi1}
holds, each component of $\H(M,\Sigma)$ is aspherical, and if $\pi_1(M)$ is
centerless, is a classifying space $K(G(M,\allowbreak \Sigma),1)$ for the
Goeritz group. As to the excluded cases in Corollary~\ref{coro:inf_pi1}, a
nilmanifold is a $3$-manifold that is a quotient of Heisenberg space by a
torsion-free lattice (topologically these are the $S^1$-bundles over the
torus with nonzero Euler class), and an infranilmanifold is a finite
quotient of a nilmanifold. Non-Haken infranilmanifolds are Seifert-fibered
with base orbifold a $2$-sphere with three cone points of types $(2,4,4)$,
$(2,3,6)$, or $(3,3,3)$. If the components of $\Diff(M)$ turn out to be
homotopy equivalent to $S^1$ for these manifolds, as expected
(see~\cite{McCullough-Soma}), then Corollary~\ref{coro:inf_pi1} will hold
without exclusion.

Corollary~\ref{coro:inf_pi1} applies whenever the (Hempel) distance
$d(M,\Sigma)$ is greater than~$3$. Combined with various results from the
literature, this provides a rather complete description of the homotopy
type of $\H(M,\Sigma)$ for this case:
\begin{maincorollary}\label{coro:highdistance}
If $d(M,\Sigma)>3$ then $\H(M,\Sigma)$ has finitely many components, each
of which is contractible. In fact, the number of components of
$\H(M,\Sigma)$ equals $|\Mod(M)|/|\Mod(M,\Sigma)|$, and if
$d(M,\Sigma)>2\genus(\Sigma)$, then $\H(M,\Sigma)$ is contractible.
\par
\end{maincorollary}

When $\pi_1(M)$ is finite, $\Diff(M)$ and $\H(M,\Sigma)$ can have more
interesting homotopy types. For these cases, $M$ admits an elliptic
structure, that is, a Riemannian metric of constant sectional curvature
$1$, or equivalently $M$ is a quotient of the standard round $3$-sphere by
a group of isometries acting freely. For elliptic $3$-manifolds, the
(Generalized) Smale Conjecture asserts that the inclusion $\Isom(M)\to
\Diff(M)$ of the subgroup of isometries of $M$ is a homotopy
equivalence. As we will discuss in Section~\ref{sec:smaleconj} below, the
Smale Conjecture is known for some cases, including $S^3$ and lens spaces
other than $\RP^3$, but is open in general. Our computations of
$\H_i(M,\Sigma)$ require this homotopy equivalence, and therefore must be
regarded as modulo the Smale Conjecture for the unknown cases. In the
statements of our remaining results, $C_2$ denotes a cyclic group of
order~$2$.
\begin{maintheorem} For $n\geq 0$ let $\Sigma_n$ be the unique Heegaard 
surface of genus $n$ in $S^3$.
\begin{enumerate}
\item[(1)] $\H(S^3,\Sigma_0)\simeq \RP^3$.
\item[(2)] $\H(S^3,\Sigma_1)\simeq \RP^2\times \RP^2$.
\item[(3)] For $n\geq 2$, $\H_i(S^3,\Sigma_n)\cong \pi_i(S^3\times S^3)$ for
$i\geq 2$, and there is a non-split exact sequence
\[ 1\to C_2 \to \H_1(S^3,\Sigma_n) \to G(S^3,\Sigma_n)\to 1\ .\]
\end{enumerate}
\label{thm:S3}
\end{maintheorem}

\begin{maintheorem} Let $L=L(m,q)$ be a lens space with $m\geq 2$ and $1\leq
q\leq m/2$. Assume, if necessary, that $L$ satisfies the Smale Conjecture.
For $n\geq 1$, let $\Sigma_n$ be the unique Heegaard surface of genus~$n$
in $L$.
\begin{enumerate}
\item If $q\geq 2$, then
\begin{enumerate}
\item $\H(L,\Sigma_1)$ is contractible.
\item For $n\geq 2$, $\H_i(L,\Sigma_n)=0$ for $i\geq 2$, and there is an
exact sequence
\[ 1\to \Z\times \Z\to \H_1(L,\Sigma_n) \to G(L,\Sigma_n)\to 1\ .\]
\end{enumerate}
\item If $m>2$ and $q=1$, then
\begin{enumerate}
\item $\H(L,\Sigma_1)\simeq \RP^2$.
\item For $n\geq 2$, $\H_i(L,\Sigma_n)\cong \pi_i(S^3)$ for $i\geq 2$, and
there are exact sequences
\[ 1\to \Z\to \H_1(L,\Sigma_n) \to G(L,\Sigma_n)\to 1\]
for $m$ odd, and
\[ 1\to \Z\times C_2\to \H_1(L,\Sigma_n) \to G(L,\Sigma_n)\to 1\]
for $m$ even.
\end{enumerate}
\item If $L=L(2,1)$, then
\begin{enumerate}
\item $\H(L,\Sigma_1)\simeq \RP^2\times \RP^2$.
\item For $n\geq 2$, $\H_i(L,\Sigma_n)\cong \pi_i(S^3\times S^3)$ 
for $i\geq 2$, and there is an exact sequence
\[ 1\to C_2\times C_2\to \H_1(L,\Sigma_n) \to G(L,\Sigma_n)\to 1\ .\]
\end{enumerate}
\end{enumerate}
\label{thm:lens_spaces}
\end{maintheorem}

\begin{maintheorem} Let $E$ be an elliptic $3$-manifold, but not $S^3$ or a lens
space. Assume, if necessary, that $E$ satisfies the Smale Conjecture.  Let
$\Sigma$ be a Heegaard surface in $E$.
\begin{enumerate}
\item If $\pi_1(E)\cong D^*_{4m}$, or if $E$ is one of the three manifolds
with fundamental group either $T^*_{24}$, $O^*_{48}$, or $I^*_{120}$, then
$\H_i(E,\Sigma)\cong\pi_i(S^3)$ for $i\geq 2$ and there is an exact
sequence
\[ 1\to C_2\to \H_1(E,\Sigma) \to G(E,\Sigma)\to 1\ .\]
\item If $E$ is not one of the manifolds in Case~(1), that
is, either $\pi_1(E)$ has a nontrivial cyclic direct factor, or $\pi_1(E)$
is a diagonal subgroup of index~$2$ in $D^*_{4m}\times C_n$ or of index~$3$
in $T^*_{48}\times C_n$, then $\H_i(E,\Sigma)=0$ for $i\geq 2$, and
there is an exact sequence
\[ 1\to \Z\to \H_1(E,\Sigma) \to G(E,\Sigma)\to 1\ .\]
\end{enumerate}
\label{thm:elliptics}
\end{maintheorem}

Theorems~\ref{thm:genus2case} and \ref{thm:S1timesS2} are proven in
Sections~\ref{sec:highgenus} and \ref{sec:S1timesS2} respectively, and
Corollaries~\ref{coro:inf_pi1} and~\ref{coro:highdistance} in
Section~\ref{sec:irreducible}. Theorem~\ref{thm:S3} part (1) is proven in
Section~\ref{sec:genus0}. Theorem~\ref{thm:S3} part (2) and the (a) parts of
Theorem~\ref{thm:lens_spaces} are proven as Theorem~\ref{thm:genus1_cases}
in Section~\ref{sec:lens}, and
Theorem~\ref{thm:S3} part (3), the (b) parts of Theorem~\ref{thm:lens_spaces} as
Theorem~\ref{thm:high-genus_cases} in
Section~\ref{sec:elliptics}, along with
Theorem~\ref{thm:elliptics}. The other
sections provide auxiliary material used in the proofs.

The authors are grateful to the referee for thoughtful suggestions that
improved the manuscript.

\section{Spaces of images, mapping class groups, and Goeritz groups}
\label{sec:intro}

In this section and the next, we assume only that $M$ is a closed manifold
and $\Sigma$ is a closed submanifold of positive codimension (although much
of what we say extends to more general contexts). 

A submanifold $\Sigma'$ of $M$ is called an \textit{image} of $\Sigma$ if
there is a diffeomorphism of $M$ carrying $\Sigma$ onto $\Sigma'$. The
images of $\Sigma$ correspond to the left cosets
$\Diff(M)/\Diff(M,\Sigma)$, since if $g,h\in\Diff(M)$, then
$g(\Sigma)=h(\Sigma)$ if and only if $g^{-1}h\in \Diff(M,\Sigma)$.
Therefore we call $\Diff(M)/\Diff(M,\Sigma)$ the \textit{space of images
  equivalent to $\Sigma$,} and denote it by $\Img(M,\Sigma)$. In
particular, when $(M,\Sigma)$ is a Heegaard splitting of a closed
$3$-manifold, $\Img(M,\Sigma)$ is the space of Heegaard splittings
v$\mathcal{H}(M,\Sigma)$ as defined in the introduction.

A Fr\'echet space is a complete metrizable locally convex topological
vector space. The topology of a Fr\'echet space is defined by a countable
collection of seminorms such that $f_j\to f$ if $\|f_j-f\|\to 0$ for each
of the seminorms. A Fr\'echet manifold is a (usually infinite-dimensional)
manifold locally modeled on open subsets of a Fr\'echet space, with smooth
(as maps of the Fr\'echet space) transition functions. Two convenient
references for Fr\'echet spaces and Fr\'echet manifolds are
R. Hamilton~\cite{Hamilton} and A. Kriegl and
P. Michor~\cite{Kriegl-Michor}.

The space of images $\Img(M,\Sigma)$ is a Fr\'echet manifold locally
modeled on the sections close to the zero section from $\Sigma$ to its
normal bundle~\cite[Example 4.1.7]{Hamilton}. It follows that
$\Img(M,\Sigma)$ has the homotopy type of a CW-complex (see for example
Section~2.1 of~\cite{HKMR}).

The \textit{mapping class group $\Mod(M,\Sigma)$ of the pair $(M,\Sigma)$}
is defined to be the discrete group $\Diff(M,\Sigma)/\diff(M,\Sigma)$,
where $\diff(M,\Sigma)$ (and in general, any space of isometries,
diffeomorphisms, or imbeddings whose name begins with a small letter) is
the connected component of the identity diffeomorphism. In particular, when
$\Sigma$ is empty, we write this as $\Mod(M)$ and it becomes the usual
mapping class group. Note that we allow orientation-reversing
diffeomorphisms, when $M$ is orientable, so our $\Mod(M,\Sigma)$ is what is
often called the extended mapping class group. 

The \textit{Goeritz group} of the pair $(M,\Sigma)$ is the kernel
v$G(M,\Sigma)$ of the natural map $\Mod(M,\Sigma)\to \Mod(M)$. When $\Sigma$
has codimension~$1$ and is two-sided in $M$, the \textit{pure} Goeritz
group $G_0(M,\Sigma)$ is defined to consist of the elements of
$G(M,\Sigma)$ that do not interchange the sides of~$\Sigma$. It is a
subgroup of index at most~$2$ in~$G(M,\Sigma)$.

To indicate the subgroup of orientation-preserving, we use a ``+''
subscript, as in $\Diff_+(M)$ or $\Isom_+(S^3)$.

\section{Fibration theorems}
\label{sec:main}

In this section, we will obtain fibrations using a method of
R. Palais~\cite{Palais} and J. Cerf~\cite{Cerf}, which is based on the
following definition. Let $X$ be a $G$-space and $x_0\in X$. A
\textit{local cross-section} for $X$ at $x_0$ is a map $\chi$ from a
neighborhood $U$ of $x_0$ into $G$ such that $\chi(u)x_0= u$ for all $u\in
U$. By replacing $\chi(u)$ by $\chi(u)\chi(x_0)^{-1}$, one may always
assume that $\chi(x_0)= 1_G$. If $X$ admits a local cross-section at each
point, it is said to admit local cross-sections. 

A local cross-section $\chi_0\colon U_0\to G$ at a single point $x_0$
determines a local cross-section $\chi\colon gU_0\to G$ at any point $gx_0$
in the orbit of $x_0$, by the formula $\chi(u)=g\chi_0(g^{-1}u)g^{-1}$,
since then
$\chi(u)(gx_0)=g\chi_0(g^{-1}u)g^{-1}gx_0=g\chi_0(g^{-1}u)x_0=gg^{-1}u=u$. In
particular, if $G$ acts transitively on $X$, then a local cross-section at
any point provides local cross-sections at all points.

From \cite{Palais} we have
\begin{proposition} Let $G$ be a topological group and $X$ a
$G$-space admitting local cross-sections. Then any equivariant map of
a $G$-space into $X$ is locally trivial.
\label{prop:palaisTheoremA}
\end{proposition}

\noindent In fact, when $\pi\colon Y\to X$ is $G$-equivariant, the
local coordinates on $\pi^{-1}(U)$ are just given by sending the point
$(u,z)\in U\times \pi^{-1}(x_0)$ to $\chi(u)\cdot z$.

We continue to assume only that $M$ is a closed manifold and $\Sigma$ is a
closed submanifold of positive codimension. Clearly
$\Diff(M)/\Diff(M,\Sigma)$ and $\Diff(M)/\diff(M,\Sigma)$ are
$\Diff(M)$-spaces.
\begin{theorem} 
$\Diff(M)/\Diff(M,\Sigma)$ and $\Diff(M)/\diff(M,\Sigma)$
admit local $\Diff(M)$ cross-sections.
\label{thm:cross-sections}
\end{theorem}

\begin{proof} We will argue for $\Diff(M)/\Diff(M,\Sigma)=\Img(M,\Sigma)$, 
since the case of $\Diff(M)/\diff(M,\Sigma)$ requires only trivial
modifications. Since $\Diff(M)$ acts transitively, we need only find a
local cross-section at $1_M\Diff(M,\Sigma)$.

Fix a Riemannian metric on $M$ and a tubular neighborhood $N(\Sigma)$
determined by the exponential map $\Exp\colon \nu_{<\epsilon}(\Sigma)\to
N(\Sigma)\subset M$, where $\nu_{<\epsilon}(\Sigma)$ is the space of normal
vectors of $\Sigma$ of length less than~$\epsilon$. For all $g$ in a
sufficiently small $C^\infty$-neighborhood $V$ of $1_M$ (in fact for all
$g$ sufficiently $C^1$-close to $1_M$) in $\Diff(M)$, the tangent planes to
$g(\Sigma)$ remain almost perpendicular to the tangent planes of the fibers
of $N(\Sigma)$, and consequently $g(\Sigma)$ meets each normal fiber in
$N(\Sigma)$ in exactly one point.

Denote by $\Sect(\Sigma,TM)$ the sections from $\Sigma$ to the restriction
of $TM$ to $\Sigma$, and by $Z$ the zero-section in $\Sect(\Sigma,TM)$ or
in any other space of sections.

The image $W$ of $V$ in $\Img(M,\Sigma)$ is an open neighborhood of
$1_M\Diff(M,\Sigma)$. Define $\Phi\colon W\to \Sect(\Sigma,TM)$ by putting
$\Phi(g\Diff(M,\Sigma))(x)$ equal to the unique vector in $T_xM\cap
\nu_{<\epsilon}(\Sigma)$ that exponentiates to $g(\Sigma)\cap \Exp
(\nu_{<\epsilon}(x))$, where $\nu_{<\epsilon}(x)$ is the fiber of
$\nu_\epsilon(\Sigma)$ at $x$. In particular, $\Phi(1_M\Diff(M,\Sigma))=Z$,
the zero section.

Lemma~c from \cite{Palais} provides a continuous linear map $k\colon
\Sect(\Sigma,TM)\to \Sect(M,TM)$ such that for each $X\in
\Sect(\Sigma,TM)$, $k(X)|_{\Sigma}=X$. In fact, $k$ is defined just by
using parallel translation to push each $X(x)$ to a vector at each of the
points in the normal fiber at $x$, then multiplying by a smooth function
that is $1$ on $\Sigma$ and is $0$ off of $N(\Sigma)$.

Now, define $\TExp\colon \Sect(M,TM)\to \Maps(M,M)$, the space of smooth
maps from $M$ to $M$ with the $C^\infty$-topology, by
$\TExp(X)(x)=\Exp(X(x))$. By Lemmas~a and~b of~\cite{Palais}, $\TExp$ is
continuous and maps a neighborhood of $Z$ into $\Diff(M)$. On a
neighborhood $U$ of $1_M\Diff(M,\Sigma)$ contained in $W$ and small enough
so that $\TExp\circ k\circ \Phi(U)\subset\Diff(M)$, $\TExp\circ k\circ
\Phi$ is a local cross-section. For if $g\Diff(M,\Sigma)\in U$, then by
definition of $\Phi$ we have
\[\Exp\circ \Phi(g\Diff(M,\Sigma))(x)=g(\Sigma)\cap \Exp(\nu_{<\epsilon}(x))\]
for each $x\in \Sigma$. Therefore
\[\Exp\circ \Phi(g\Diff(M,\Sigma))\, 1_M\Diff(M,\Sigma)(\Sigma)
=\Exp\circ \Phi(g\Diff(M,\Sigma))(\Sigma)=g(\Sigma)\]
and consequently $\TExp\circ k\circ
\Phi(g\Diff(M,\Sigma))1_M\Diff(M,\Sigma)=g\Diff(M,\Sigma)$.
\end{proof}

Proposition~\ref{prop:palaisTheoremA} and Theorem~\ref{thm:cross-sections}
give immediately 
\begin{corollary}\label{coro:basic_fibrations}
The quotient maps $\Diff(M)\to
\Diff(M)/\Diff(M,\Sigma)$ and $\Diff(M)\to
\Diff(M)/\diff(M,\Sigma)$ are fibrations.
\end{corollary}
\noindent Also, the natural map $\Diff(M)/\diff(M,\Sigma)\to
\Diff(M)/\Diff(M,\Sigma)$ is $\Diff(M)$-equivariant, with fiber the
discrete group $\Diff(M,\Sigma)/\diff(M,\Sigma)\allowbreak
=\Mod(M,\Sigma)$, so we have
\begin{corollary} The natural map
\[ \Diff(M)/\diff(M,\Sigma)\to \Diff(M)/\Diff(M,\Sigma)=\Img(\Sigma)\]
is a covering map with fiber $\Mod(M,\Sigma)$.
\label{coro:fibration}
\end{corollary}

\begin{corollary} For $i\geq 2$, $\pi_i(\Diff(M)/\diff(M,\Sigma))\to
\pi_i(\Img(\Sigma))$ is an isomorphism, and there is an exact sequence
\begin{gather*}
1\to \pi_1(\Diff(M)/\diff(M,\Sigma))\to \pi_1(\Img(M,\Sigma))\\
\to \Mod(M,\Sigma) \to \pi_0(\Diff(M)/\diff(M,\Sigma))\to 
\pi_0(\Img(M,\Sigma))\to 1\ .
\end{gather*}\par
\label{coro:exactseq}
\end{corollary}

For later use, we include the following lemma.
\begin{lemma}
The map $\Diff(M,\Sigma)\to \Diff(\Sigma)$ defined by
restriction is a fibration over its image (which is a union of path
components of $\Diff(\Sigma)$).
\label{lem:restrictToSigma}
\end{lemma}

\begin{proof}
Let $\Imb(\Sigma,M)$ be the space of all imbeddings of $\Sigma$ into $M$
that extend to diffeomorphisms of $M$. From \cite{Palais}, the map
$\rho\colon \Diff(M)\to \Imb(\Sigma,M)$ defined by $\rho(f)=f|_\Sigma$ is a
fibration. We identify the image of $\Diff(M,\Sigma)\to \Diff(\Sigma)$ with
the subspace of elements of $\Imb(\Sigma,M)$ that take $\Sigma$ to
$\Sigma$. Since $\Diff(M,\Sigma)$ is the full preimage of this subspace,
over its image $\Diff(M,\Sigma)\to \Diff(\Sigma)$ is just the pullback
fibration.
\end{proof}

\section{Heegaard splittings of genus at least $2$}
\label{sec:highgenus}

This section contains the proof of Theorem~\ref{thm:genus2case}. We will
use the following theorem of
A. Hatcher~\cite{HatcherHaken,HatcherIncompressible} and
N. Ivanov~\cite{I3,I4}:
\begin{theorem}[Hatcher, Ivanov]\label{thm:hatcher}
Let $M$ be a Haken $3$-manifold.
\begin{enumerate}
\item[(i)] If $\partial M\neq \emptyset$, then 
$\diff(M\rel \partial M)$ is contractible.
\item[(ii)] If $M$ is closed, then there is a 
homotopy equivalence $(S^1)^k\to \diff(M)$, where $k$ is the rank
of the center of~$\pi_1(M)$.
\end{enumerate}
\end{theorem}
\noindent In \cite{HatcherHaken}, the results are stated for PL
homeomorphisms, but the Smale Conjecture for $S^3$, also proven by
Hatcher~\cite{HatcherSmale}, extends the results to the smooth
category (see~\cite{HatcherIncompressible}). For
Theorem~\ref{thm:genus2case}, we will only need part~(i), but part~(ii)
will be used later.

\begin{generalgenus2theorem} Suppose that $\Sigma$ has genus at least 
$2$. Then $\pi_q(\Diff(M))\to \H_q(M,\Sigma)$ is an isomorphism for 
$q\geq 2$, and there are exact sequences
\begin{gather*}
1\to \pi_1(\Diff(M))\to \H_1(M,\Sigma)\to G(M,\Sigma) \to 1\ ,\\
1\to G(M,\Sigma)\to \Mod(M,\Sigma)\to \Mod(M)\to \H_0(M, \Sigma)\to 1\ .
\end{gather*}\par
\end{generalgenus2theorem}

\begin{proof}
Since the genus of $\Sigma$ is at least $2$, $\diff(\Sigma)$ is
contractible \cite{Earle-Eells}. From Lemma~\ref{lem:restrictToSigma},
there is a fibration
\[ \Diff(M\rel \Sigma)\cap \diff(M,\Sigma) \to 
\diff(M,\Sigma)\to \diff(\Sigma)\ .\] Any two elements of $\Diff(M\rel
\Sigma)\cap \diff(M,\Sigma)$ are isotopic preserving $\Sigma$. Since
$\pi_1(\diff(\Sigma))$ is trivial, they are isotopic relative to
$\Sigma$. Therefore $\Diff(M\rel \Sigma)\cap \diff(M,\Sigma)=
\diff(M\rel\Sigma)$, which is contractible using
Theorem~\ref{thm:hatcher}(i), so the fibration shows that
$\diff(M,\Sigma)$ is contractible.

By Corollary~\ref{coro:basic_fibrations}, the quotient map 
\[\Diff(M)\to \Diff(M)/\diff(M,\Sigma)\] 
is a fibration. Since it has contractible fiber, it is a homotopy
equivalence. The assertions of Theorem~\ref{thm:genus2case}
now follow from Corollary~\ref{coro:exactseq}.
\end{proof}

\section{The case of $S^1\times S^2$}
\label{sec:S1timesS2}

In this section, we will prove Theorem~\ref{thm:S1timesS2}. For more
concise notation, we write $M$ for $S^1\times S^2$. In addition, we write
the standard $2$-sphere $S^2$ as $D^2_+\cup D^2_-$, the upper and lower
hemispheres, $N$ or $+N$ for the center point of $D^2_+$, the north pole,
and $-N$ for the south pole. The isometry group of $S^2$ is the orthogonal
group $\O(3)$. By $\O(2)$ we denote the $\O(2)$-subgroup of $\SO(3)$ that
preserves $D^2_+\cap D^2_-$; its subgroup $\SO(2)$ preserves each of
$D^2_+$ and $D^2_-$, while elements of $\O(2)-\SO(2)$ interchange $D^2_+$
and $D^2_-$. Since $\SO(3)$ acts transitively on $S^2$ and the stabilizer
of $N$ is $\SO(2)$, the space of cosets $\SO(3)/\SO(2)$ is homeomorphic to
$S^2$.

In $M$ define $T=S^1\times D^2_+\cap S^1\times D^2_-$. It is a Heegaard
surface in $M$, and the resulting splitting is called the \textit{standard}
genus-$1$ Heegaard splitting of $M$. The following must be well known, but
we include a proof here.
\begin{proposition}\label{prop:uniqueS1timesS2splitting}
Up to isotopy $M$ has a unique Heegaard splitting for each positive genus.
\end{proposition}
\begin{proof}
Assume first that the Heegaard splitting has genus $1$. By Haken's
Lemma~\cite{Haken} (see also~\cite[Lemma 1.1]{Casson-Gordon}), there is a
$2$-sphere $S$ in $M$ that meets each of the solid tori of the splitting in
a single disk. It is easy to check that $M$ contains a unique essential
$2$-sphere up to isotopy, so we may assume that $S$ is a fiber and each
solid torus of the splitting is a regular neighborhood of a loop crossing
$S$ in a single point. By the well-known light-bulb trick, such a loop is
isotopic to a loop of the form $S^1\times \{x\}$, so the Heegaard splitting
is isotopic to the standard one.

Suppose now that the Heegaard splitting has genus $n>1$, and apply Haken's
Lemma as before to obtain a sphere that intersects each handlebody in a
disk. Compressing the splitting along one of the two disks, then removing a
neighborhood of the essential sphere, one of the handlebodies becomes a
handlebody of genus $n-1$, and the other a handlebody with two
punctures. Filling in the punctures gives a Heegaard splitting of $S^3$ of
genus $n-1$. Waldhausen~\cite{Waldhausen} showed that every positive genus
Heegaard splitting of $S^3$ is a stabilization, which implies that the
original Heegaard splitting of $M$ was a stabilization. Inductively, the
original splitting is obtained by repeated stabilization of the standard
genus-$1$ splitting.
\end{proof}
\noindent Proposition~\ref{prop:uniqueS1timesS2splitting} shows, of course,
that $\H(M,\Sigma)$ is connected for every Heegaard splitting of $M$.

Our proof of Theorem~\ref{thm:S1timesS2} will use the description of
$\Diff(M)$ due to A. Hatcher \cite{HatcherS1timesS2,
  HatcherS1timesS2new}. To set notation, define $R(M)$ to be the subgroup
of $\Diff(S^1\times S^2)$ consisting of the diffeomorphisms that take each
$\{x\}\times S^2$ to some $\{y\}\times S^2$ by an element of the orthogonal
group $\O(3)$ that depends on $x$, and where the diffeomorphism of $S^1$
vsending each $x$ to the corresponding $y$ is an element of~$\O(2)$.

As noted in \cite{HatcherS1timesS2, HatcherS1timesS2new}, $R(M)$ is
homeomorphic (although not isomorphic) to the subgroup $\O(2)\times
\O(3)\times \Omega\SO(3)\subset \Diff(M)$, where $\Omega\SO(3)$ denotes the
space of smooth loops $\gamma\colon S^1\to \SO(3)$ taking the basepoint
$0\in S^1=\R/\Z$ to the identity rotation. The $\O(2)$-coordinate tells the
effect of an element of $R(M)$ on the $S^1$-coordinate of $S^1\times S^2$, the
$\O(3)$-coordinate tells the effect on the $S^2$-coordinate of $\{0\}\times
S^2$, and the element of $\Omega\SO(3)$ tells the deviation from being
constant in the $S^2$-coordinate as the $S^1$-coordinate varies. More
precisely, an element $(f,g,\gamma)\in R(M)$ acts on $M$ by sending
$(t,x)\in S^1\times S^2$ to $(f(t),\gamma(t)(g(x)))$.

\begin{theorem}[A. Hatcher]\label{thm:DiffS1timesS2} The inclusion $R(M)\to
\Diff(M)$ is a homotopy equivalence.
\end{theorem}

Let $R(M,T)$ be the subgroup of $R(M)$ that takes $T$ to $T$, that is,
$R(M)\cap \Diff(M,T)$. Under the homeomorphism from $R(M)$ to $\O(2)\times
\O(3)\times \Omega\SO(3)$, $R(M,T)$ corresponds to the subgroup
$\O(2)\times (C_2\times \O(2))\times \Omega\SO(2)$. The $C_2$-factor of
$C_2\times \O(2)$ is generated by the reflection through the
equator~$D^2_+\cap D^2_-$.

\begin{proposition}\label{prop:easyHatcher} The inclusion
$R(M,T)\to \Diff(M,T)$ is a homotopy equivalence.
\end{proposition}

\begin{proof} 
By Lemma~\ref{lem:restrictToSigma}, the restriction map $\Diff(M,T)\to
\Diff(T)$ is a fibration over its image, which we will denote by
$\Diff_0(T)$. Letting $R(T)$ denote the diffeomorphisms of $T=S^1\times
(D^2_+\cap D^2_-)$ that send each $\{x\}\times (D^2_+\cap D^2_-)$ to some
$\{y\}\times (D^2_+\cap D^2_-)$ by an element of $\O(2)$, and such that
sending each $\{x\}$ to the corresponding $\{y\}$ is an element of $\O(2)$,
we have a restriction map $R(M,T)\to R(T)$ that is a $2$-fold covering
projection.

We now have a commutative diagram
\begin{equation*}
\begin{CD}
R(M \rel T) @>>>  R(M,T) @>>> R(T)\\
@VVV @VVV @VV{j}V\\
\Diff(M \rel T) @>>>  \Diff(M,T) @>>> \Diff_0(T)\\
\end{CD}
\end{equation*}%
whose rows are fibrations and vertical maps are inclusions. The two
components of $\Diff(M \rel T)$ are contractible, using
Theorem~\ref{thm:hatcher}, so the first vertical arrow is a homotopy
equivalence. To complete the proof, it suffices to check that the third
vertical arrow $j$ is a homotopy equivalence.

Note first that $R(T)$ is homeomorphic to $\O(2)\times \O(2)\times \Omega
\SO(2)$, compatibly with our homeomorphism from $R(M,T)$ to $\O(2)\times
(C_2\times \O(2))\times \Omega \SO(2)$. A diffeomorphism of $T$ lies in
$\Diff_0(T)$ exactly when it preserves the circles $\{t\}\times (D^2_+\cap
D^2_-)$ up to isotopy. These are exactly the diffeomorphisms isotopic to
elements of $R(T)$, so $j$ is surjective on path components. Since elements
in different path components of $R(T)$ induce distinct outer automorphisms
of $\pi_1(T)$, $j$ is injective on path components. The composition of
inclusions $\SO(2)\times \SO(2)\to r(T)\to \diff(T)$ is a well-known
homotopy equivalence (see for example A.~Gramain~\cite{Gramain}). The
components of $\Omega \SO(2)$ are contractible, so the inclusion
$\SO(2)\times \SO(2)\to r(T)$ is a homotopy equivalence as well. Therefore
$r(T)\to \diff(T)$ is a homotopy equivalence, and it follows that $j$ is a
homotopy equivalence on every path component of~$R(T)$.
\end{proof}

\begin{S1S2theorem} For the unique genus-$1$ Heegaard surface $\Sigma$
in $S^1\times S^2$, $\H(S^1\times S^2,\Sigma)$ is homotopy equivalent to
$LS^2/\langle \alpha \rangle$, where $\alpha$ is the involution induced by
the antipodal map of~$S^2$.
\end{S1S2theorem}

\begin{proof}
By Proposition~\ref{prop:uniqueS1timesS2splitting}, we may use $\Sigma=T$
as our genus-$1$ Heegaard surface.

We have a commutative diagram whose vertical arrows are inclusions:
\begin{equation*}
\begin{CD}
R(M,T) @>>>  R(M) @>>> R(M)/R(M,T)\\
@VVV @VVV @VVV\\
\Diff(M,T) @>>>  \Diff(M) @>>> \Diff(M)/\Diff(M,T)\\
\end{CD}
\end{equation*}%
By Corollary~\ref{coro:basic_fibrations}, the bottom row is a fibration. We
claim that the top row is also a fibration. Since $R(M)$ acts transitively
on $R(M)/R(M,T)$, it suffices to construct a local $R(M)$ cross-section at
the coset $1_M\,R(M,T)$.

We will write $X$ for $S^1\times \{\pm N\}$, a union of two circles in
$M$. Since $R(M,T)$ is exactly the subgroup of $R(M)$ that leaves $X$
invariant, the image $r(X)$ of $X$ under a coset $rR(M,T)$ is well-defined,
and $rR(M,T)=sR(M,T)$ if and only if $r(X)=s(X)$.

For $w\in S^2-\{-N\}$, let $\rho_w\in \SO(3)$ be the unique rotation with
axis the cross product $w\times N$ that rotates $w$ to $N$, and let
$\rho_N$ be the identity rotation. Now let $U$ be the open set in
$R(M)/R(M,T)$ consisting of the elements $rR(M,T)$ such that $r(X)\cap
T=\emptyset$. When $rR(M,T)\in U$, 
each component of $r(X)$ is contained in either the
interior of $S^1\times D^2_+$ or the interior of~$S^1\times D^2_-$, and
$r(X)$ meets each $\{t\}\times D^2_+$ in a single point.

To define $\chi\colon U\to R(M)$, let $rR(M,T)\in U$, $r=(f,g,\gamma)$. For
each $t\in S^1$, put $w_t=r(X)\cap (\{t\}\times D^2_+)$, $g_0=\rho_{w_0}$,
and $\delta(t)=\rho_{g_0(w_t)}$ (note that $g_0(w_t)\neq -N$, since this
would say that $w_t=\rho_{w_0}^{-1}(-N)=-\rho_{w_0}^{-1}(N)=-w_0\in
M-(S^1\times D^2_+)$). Since $\delta(0)=\rho_{g_0(w_0)}=\rho_N=1_{S^2}$,
$\delta\in \Omega\SO(3)$ and we can define $\chi(r)=(1,g_0,\delta)^{-1}$.
To verify that $\chi$ is a local cross-section, we have
$\chi(r)^{-1}(t,w_t)=(1,g_0,\delta)(t,w_t)=(t,\delta(t)(g_0(w_t)))=
(t,\rho_{g_0(w_t)}(g_0(w_t)))=(t,N)$, so $\chi(r)^{-1}r\in R(M,T)$. That
is, $\chi(r)(1_Mr(M,T))=rR(M,T)$, completing the proof of the claim.

By Theorem~\ref{thm:DiffS1timesS2} and Proposition~\ref{prop:easyHatcher},
the first and second vertical arrows of the diagram are homotopy
equivalences. Therefore the third is a (weak) homotopy equivalence.  To
complete the proof, we will construct a homeomorphism $\phi\colon
R(M)/R(M,T)\to LS^2/\langle\alpha\rangle$.

Define $\phi(rR(M,T))$ to be the element represented by the loop $\gamma$
defined $\gamma(t)=\proj_{S^2}r(t,N)$. Note that although
$\proj_{S^2}r(t,N)$ is not well-defined on cosets as an element of $LS^2$,
it is well-defined in $LS^2/\langle \alpha\rangle$, and clearly $\phi$ is
continuous. Injectivity of $\phi$ follows using the fact that $R(M,T)$ is
exactly the subgroup of $R(M)$ that preserves $S^1\times\{\pm N\}$.

For surjectivity, it suffices to show that if $\tau\colon S^1\to S^2$ is a
smooth loop, then there exists $r_\tau\in R(M)$ such that
$r_\tau(t,\tau(t))=(t,N)$, since then we have $\phi(r_\tau^{-1}
R(M,T))=\tau$. To show $r_\tau$ exists, we will apply a sequence of
elements of $R(M)$ whose composition moves each $(t,\tau(t))$ to~$(t,N)$.

First, there is an element $r=(1,g,1)\in R(M)$ such that
that $r(0,\tau(0))=(0,N)$, so we may assume that $\tau(0)=N$. Next,
there exist $0<\epsilon<1/2$ and an element of the form $r=(1,1,\gamma)$
such that $r(t,\tau(t))=(t,N)$ for $t\in [-\epsilon, \epsilon]\subset
S^1$; for $t\in [-\epsilon, \epsilon]$,
$r(t,x)=(t,\rho_{\tau(t)}(x))$, where $\rho_w$ is as
defined earlier in the proof where we were constructing a local $R(M)$
cross-section for $R(M)\to R(M)/R(M,T)$. So we may assume that
$\tau(t)=N$ for $t\in [-\epsilon,\epsilon]$.

Regard $\tau$ as a path $I\to S^1\to S^2=\SO(3)/\SO(2)$. By the homotopy
lifting property, $\tau$ lifts to a path $\delta \colon I\to \SO(3)$ with
$\delta(0)=1_{\SO(3)}$ and $\delta(t)(N)=\tau(t)$. In particular,
$\delta(t)(N)=N$ for $t\in [0,\epsilon]\cup [1-\epsilon, 1]$ so
$\delta(t)\in \SO(2)$ for these $t$. Changing $\delta(t)$ by a smooth
isotopy supported on $[0,\epsilon/2]\cup [1-\epsilon/2,1]$, we may assume
that $\delta(t)=1_{S^2}$ for $t\in [0,\epsilon/2]\cup [1-\epsilon/2,
1]$. Consequently, $\delta$ defines an element $\delta\colon S^1\to \SO(3)$
of the smooth loop space $\Omega \SO(3)$. Putting $r(t,x)=(t,
\delta(t)(x))$, we have $\phi(rR(M,T))(t) = \delta(t)(N)= \tau(t)$.
\end{proof}

\section{The irreducible case}
\label{sec:irreducible}

In this section, we will prove Corollaries~\ref{coro:inf_pi1}
and~\ref{coro:highdistance}. For the manifolds in
Corollary~\ref{coro:inf_pi1}, the center $Z(\pi_1(M))$ is $\Z^k$ where $k=3$
when $M$ is the $3$-torus and $k$ is $0$ or $1$ otherwise. Moreover,
$\diff(M)\simeq (S^1)^k$; for Haken manifolds this is
Theorem~\ref{thm:hatcher}(ii) above, and for hyperbolic $M$, $k=0$ and it
is D. Gabai's result~\cite{Gabai} that the components of $\Diff(M)$ are
contractible. When $M$ is non-Haken and not hyperbolic, it is
Seifert-fibered over a $2$-orbifold $O$ of nonpositive (orbifold) Euler
characteristic $\chi^{orb}(O)$. When $\chi^{orb}(O)<0$, that is, when $M$
has an $\widetilde{\text{SL}}(2,\mathbb{R})$ or $\mathbb{H}^2\times S^1$
geometric structure (see~\cite{Scott}), $\diff(M)\simeq (S^1)^k$
by~\cite{McCullough-Soma}. When $\chi^{orb}(O)=0$, $M$ may be Haken,
including all cases when $M$ has a Euclidean geometric structure, or it may
be a non-Haken infranilmanifold, excluded by hypothesis. In all the
non-excluded cases, the isomorphism $\pi_1(\diff(M))\to \Z^k$ is given
explicitly by taking the trace at a basepoint of $M$ of an isotopy from
$1_M$ to $1_M$ that represents a given element of~$\pi_1(\diff(M))$.

\begin{firstcorollary} Suppose that $M$ is irreducible and
$\pi_1(M)$ is infinite, and that $M$ is not a non-Haken infranilmanifold.
Then $\H_i(M,\Sigma)=0$ for $i\geq 2$, and there is an exact sequence
\[ 1\to Z(\pi_1(M)) \to \H_1(M,\Sigma) \to G(M, \Sigma)\to 1\ .\]\par
\end{firstcorollary}

\begin{proof}
All Heegaard splittings of $M$ have genus at least $2$, so we can apply
Theorem~\ref{thm:genus2case}. For $i\geq 2$, $\H_i(M,\Sigma)\cong
\pi_i(\Diff(M))$, which is $0$ since $\diff(M)\simeq (S^1)^k$,
and there is an exact sequence
\[1\to \pi_1(\Diff(M))\to \H_1(M,\Sigma)\to G(M,\Sigma) \to 1\ .\]
\end{proof}

We remark that in general, the exact sequence in
Corollary~\ref{coro:inf_pi1} need not split. Suppose that $M$ fibers over
$S^1$ with fiber $F$ and monodromy a diffeomorphism $h\colon F\to F$ of
even order~$n$, having at least two fixed points $p$ and $q$. Let $D_p$ and
$D_q$ be disjoint $h$-invariant disks about $p$ and $q$
respectively. Regard $M$ as $F\times I/\mbox{$\sim$}$ where $(x,1)\sim
(h(x),0)$. The diffeomorphisms $\widetilde{\phi}_t\colon F\times \R\to
F\times \R$ defined by $\widetilde{\phi}_t(x,s)=(x,s+nt)$ induce
diffeomorphisms $\phi_t\colon M\to M$ that are an isotopy from $1_M$ to
$1_M$ with trace a primitive element of $Z(\pi_1(M))$. Now, let $V$ be
$\overline{F\times [0,1/2]-D_q\times [0,1/2]}\cup D_p\times [1/2,1]$ and
$W$ be $\overline{F\times [1/2,1]-D_p\times [1/2,1]}\cup D_q\times
[0,1/2]$. These form a Heegaard splitting of $M$ such that
$\phi_{r/n}(V)=V$ for each integer $r$ with $1\leq r\leq n$. The loop
sending $t$ to $\phi_{t/n}$ for $0\leq t\leq 1$ represents an element
$\gamma$ of $\H_1(M,\Sigma)$ such that $\gamma^n$ is a generator $\sigma$
of $Z(\pi_1(M))\cong \Z$. If the exact sequence splits, then
$\H_1(M,\Sigma)$ is a semidirect product $\Z\rtimes G(M,\Sigma)$. This maps
surjectively onto $\Z/2\times G(M,\Sigma)$, and $(\sigma,1)$ would be an
even power in this quotient, which is impossible.

\begin{secondcorollary}
If $d(M,\Sigma)>3$ then $\H(M,\Sigma)$ has finitely many components, each
of which is contractible. In fact, the number of components of
$\H(M,\Sigma)$ equals $|\Mod(M)|/|\Mod(M,\Sigma)|$, and if
$d(M,\Sigma)>2\genus(\Sigma)$, then $\H(M,\Sigma)$ is contractible.
\par
\end{secondcorollary}

\begin{proof}
All splittings of reducible $3$-manifolds have distance $d(M,\Sigma)=0$,
and by J.~Hempel~\cite{Hempel} and A. Thompson~\cite{Thompson}),
$d(M,\Sigma)>2$ implies that $M$ is atoroidal and not Seifert-fibered, so
$M$ is hyperbolic.  Corollary~\ref{coro:inf_pi1} shows that each component
of $\H(M,\Sigma)$ is a $K(G(M,\Sigma),1)$-space. By~\cite{Johnson},
$d(M,\Sigma)>3$ implies that $\Mod(M,\Sigma)\to \Mod(M)$ is injective, so
$G(M,\Sigma)$ is trivial.  Therefore the path components of $\H(M,\Sigma)$
are contractible.

D. Gabai~\cite{Gabai} showed that the inclusion of the finite set of
isometries into $\Diff(M)$ is a homotopy equivalence, so $\Mod(M,\Sigma)$
and hence $\H_0(M,\Sigma)$ are finite. In fact, the second exact sequence
of Theorem~\ref{thm:genus2case} also shows that the number of components of
$\H(M,\Sigma)$ equals $|\Mod(M)|/|\Mod(M,\Sigma)|$.  When
$d(M,\Sigma)>2\genus(\Sigma)$, the main result of \cite{Johnson} shows that
$\Mod(M,\Sigma)\to \Mod(M)$ is also surjective, so $\H(M,\Sigma)$ is
contractible.
\end{proof}

\section{The isometries of elliptic $3$-manifolds}
\label{sec:isometries}

An elliptic $3$-manifold is a closed $3$-manifold $E$ that admits a
Riemannian metric of constant positive curvature; according to Perelman's
celebrated work, this is equivalent to $\pi_1(E)$ being finite. We always
assume that $E$ is equipped with a metric of constant curvature $1$, so is
the quotient of $S^3$ by a finite group of isometries acting freely.

The elliptic $3$-manifolds were completely classified long ago
(see~\cite{McC} for a discussion). The isometry groups of elliptic
$3$-manifolds have also been known for a long time. A detailed calculation
was given in \cite{McC}. We will have to use some of the results and
methodology of that work, so in the remainder of this section we review the
necessary parts and set up some notation.

First we recall the beautiful description of $\SO(4)$ using quaternions. A
nice reference for this is~\cite{Scott}. Fix coordinates on $S^3$ as
$\{(z_0,z_1)\;|\;z_i\in\C, z_0\overline{z_0} + z_1\overline{z_1}=1\}$. Its
group structure as the unit quaternions can then be given by writing points
in the form $z=z_0+z_1j$, where $j^2=-1$ and $jz_i=\overline{z_i}j$. The
real part $\Re(z)$ is $\Re(z_0)$, and the imaginary part $\Im(z)$ is
$\Im(z_0)+z_1j$.  The inverse of $z$ is
$\Re(z)-\Im(z)=\overline{z_0}-z_1j$. The usual inner product on $S^3$ is
given by $z\cdot w=\Re(zw^{-1})$.

The unique involution in $S^3$ is $-1$, and it generates the center of
$S^3$. The pure imaginary unit quaternions $P$ form the $2$-sphere of
vectors orthogonal to~$1$, and are exactly the elements of order~$4$.
Consequently, $P$ is invariant under conjugation by elements of
$S^3$. Conjugation induces orthogonal transformations on $P$, defining a
canonical $2$-fold covering homomorphism $S^3\to\SO(3)$ with kernel the
center.

Left multiplication and right multiplication by elements of $S^3$ are
orthogonal transformations of $S^3$, and there is a homomorphism $F\colon
S^3\times S^3\to \SO(4)$ defined by $F(z,w)(q)=z q w^{-1}$. It is
surjective and has kernel $\{(1,1),(-1,-1)\}$.  The center of $\SO(4)$ has
order~2, and is generated by $F(1,-1)$, the antipodal map of~$S^3$.

By $S^1$ we will denote the subgroup of points in $S^3$ with $z_1=0$, that
is, all $z_0\in S^1\subset \C$. Let $\xi_k=\exp(2\pi i/k)$, which generates
a cyclic subgroup $C_k\subset S^1$. The elements $S^1\cup S^1j$ form a
subgroup $\Ostar\subset S^3$, which is exactly the normalizer of $S^1$ and
of the $C_k$ with $k>2$. It is also the preimage in $S^3$ of the orthogonal
group $\O(2)\subset \SO(3)$, under the $2$-fold covering $S^3\to \SO(3)$.

When $H_1$ and $H_2$ are groups, each containing $-1$ as a central
involution, the quotient $(H_1\times H_2)/\langle (-1,-1)\rangle$ is
denoted by $H_1\ttimes H_2$. In particular, $\SO(4)$ itself is $S^3\ttimes
S^3$, and contains the subgroups $S^1\ttimes S^3$, $\Ostar\ttimes\Ostar$,
and $S^1\ttimes S^1$. The latter is isomorphic to $S^1\times S^1$, but it
is sometimes useful to distinguish between them.  Finally, $\Dih(S^1\times
S^1)$ is the semidirect product $(S^1\times S^1)\rtimes C_2$, where $C_2$
acts by complex conjugation in both factors.

There are $2$-fold covering homomorphisms
\begin{equation*}
\Ostar\times \Ostar\to
\Ostar\ttimes\Ostar\to
\O(2)\times \O(2)\to
\O(2)\ttimes\O(2)\ .
\end{equation*}
Each of these groups is diffeomorphic to four disjoint copies of the torus,
but they are pairwise nonisomorphic, as can be seen by examining their
subsets of order~$2$ elements. Similarly, $S^1\times S^3$ and $S^1\ttimes
S^3$ are diffeomorphic, but nonisomorphic.

The method used in~\cite{McC} to calculate $\Isom(E)$ is
straightforward. Let $G=\pi_1(E)$ imbedded as a subgroup of $\SO(4)$ so
that $S^3/G=E$. An element $F(z,w)$ induces an isometry on $E$ exactly when
it lies in the normalizer $\Norm(G)$ of $G$ in $\O(4)$, and this gives an
isomorphism $\Norm(G)/G\cong \Isom(E)$. So for each $G$, one just needs to
calculate $\Norm(G)$ and work out the quotient group~$\Norm(G)/G$.

For convenient reference, we include two tables from~\cite{McC}.
Table~\ref{tab:nonlens} gives the isometry groups of the elliptic
$3$-manifolds with non-cyclic fundamental group. The first column shows the
fundamental group of $E$, where $C_m$ denotes a cyclic group of order $m$,
and $D^*_{4m}$, $T^*_{24}$, $O^*_{48}$, and $I^*_{120}$ are the binary
dihedral, tetrahedral, octahedral, and icosahedral groups of the indicated
orders. The groups called index $2$ and index $3$ diagonal are certain
subgroups of $D^*_{4m}\times C_n$ and $T^*_{24}\times C_n$
respectively. Table~\ref{tab:lens} gives the isometry groups of the
elliptics with cyclic fundamental group. These are the $3$-sphere $L(1,0)$,
real projective space $L(2,1)$, and the lens spaces $L(m,q)$ with $m\geq
3$. Both tables give the full isometry group $\Isom(E)$, and the group
$\mathcal{I}(E)$ of path components of $\Isom(E)$.
\begin{table}
\begin{small}
\renewcommand{\arraystretch}{1.5}
\setlength{\tabcolsep}{2 ex}
\setlength{\fboxsep}{0pt}
\fbox{%
\begin{tabular}{l|l|l|l}
$\pi_1(E)$&$E$&$\Isom(E)$&$\mathcal{I}(E)$\\
\hline
\hline
$Q_8=D^*_8$&quaternionic&$\SO(3)\times S_3$&$S_3$\\ 
\hline 
$Q_8\times C_n$&quaternionic&$\O(2)\times S_3$&$C_2\times S_3$\\ 
\hline 
$D_{4m}^*$, $m>2$&prism&$\SO(3)\times  C_2$&$C_2$\\ 
\hline 
$D_{4m}^*\times C_n$, $m>2$&prism&$\O(2)\times C_2$&$C_2\times C_2$\\ 
\hline 
index $2$ diagonal&prism&$\O(2)\times C_2$&$C_2\times C_2$\\ 
\hline
$T_{24}^*$&tetrahedral&$\SO(3)\times C_2$&$C_2$\\ 
\hline 
$T_{24}^*\times C_n$&tetrahedral&$\O(2)\times C_2$&$C_2\times C_2$\\ 
\hline 
index $3$ diagonal&tetrahedral&$\O(2)$&$C_2$\\ 
\hline 
$O_{48}^*$&octahedral&$\SO(3)$&$\{1\}$\\ 
\hline 
$O_{48}^*\times C_n$&octahedral&$\O(2)$&$C_2$\\ 
\hline 
$I_{120}^*$&icosahedral&$\SO(3)$&$\{1\}$\\ 
\hline 
$I_{120}^*\times C_n$&icosahedral&$\O(2)$&$C_2$\\
\end{tabular}}
\end{small}
\bigskip
\caption{The isometry group $\Isom(E)$ and its group of path components
$\mathcal{I}(E)$ for the elliptic $E$ with $\pi_1(E)$ not cyclic.}
\label{tab:nonlens}
\end{table}

\begin{table}
\begin{small}
\renewcommand{\arraystretch}{1.5}
\newlength{\minipagewidth}%
\setlength{\tabcolsep}{1.5 ex}
\setlength{\fboxsep}{0pt}
\fbox{%
\begin{tabular}{l|l|l}
$m$, $q$&$\Isom(L(m,q))$&$\mathcal{I}(L(m,q))$\\
\hline
\hline
$m=1$ ($L(1,0)=S^3$)&$\O(4)$&$C_2$\\
\hline
$m=2$ ($L(2,1)=\RP^3$)&$(\SO(3)\times \SO(3))\rtimes C_2$&$C_2$\\
\hline
\settowidth{\minipagewidth}{$m>2$, $m$ even}%
\begin{minipage}{\minipagewidth}%
\noindent $m>2$, $m$ odd,\par\end{minipage} $q=1$&$\Ostar\ttimes S^3$&$C_2$\\  
\hline
$m>2$, $m$ even, $q=1$&$\O(2)\times \SO(3)$&$C_2$\\
\hline
$m>2$, $1<q<m/2$, $q^2\not\equiv\pm1\bmod{m}$&$\Dih(S^1\times S^1)$&$C_2$\\
\hline
$m>2$, $1<q<m/2$, $q^2\equiv-1\bmod{m}$&$(S^1\ttimes S^1)\rtimes C_4$&$C_4$\\
\hline
$m>2$,
\settowidth{\minipagewidth}{$\gcd(m,q+1)\gcd(m,q-1)=m$}%
\begin{minipage}[t]{\minipagewidth}%
\noindent $1<q<m/2$, $q^2\equiv1\bmod{m}$,\par
\noindent $\gcd(m,q+1)\gcd(m,q-1)=m$\rule[-1.2 ex]{0mm}{0mm}\par%
\end{minipage}%
&$\O(2)\ttimes \O(2)$&$C_2\times C_2$\\
\hline
$m>2$,
\settowidth{\minipagewidth}{$\gcd(m,q+1)\gcd(m,q-1)=2m$}%
\begin{minipage}[t]{\minipagewidth}%
\noindent $1<q<m/2$, $q^2\equiv1\bmod{m}$,\par
\noindent $\gcd(m,q+1)\gcd(m,q-1)=2m$\rule[-1.2 ex]{0mm}{0mm}\par%
\end{minipage}%
&$\O(2)\times \O(2)$&$C_2\times C_2$
\end{tabular}}
\end{small}
\bigskip
\caption{Isometry groups of elliptic manifolds $L(m,q)$ with cyclic
fundamental group.}
\label{tab:lens}
\end{table}

In $S^3$ there is a standard torus $T=\{z_0+z_1j\;|\;|z_0|=|z_1|\}$. It
bounds two solid tori, $V$ and $W$, where $|z_0|\leq |z_1|$ and
$|z_0|\geq |z_1|$ respectively. In our work, certain isometries that
preserve $T$ will be useful.
\begin{enumerate}
\item $\alpha = F(1,-1)$, the antipodal map. It preserves each of $V$ and $W$.
\item $\rho\colon z_0+z_1j\mapsto \overline{z_0}+z_1j$. It is an
orientation-reversing involution that preserves each of $V$ and $W$.
\item $\tau=F(j,j)\colon z_0+z_1j\mapsto
\overline{z_0}+\overline{z_1}j$. It is an involution that
restricts to a hyperelliptic involution
on each of $V$ and $W$.
\item $\sigma_+=F(i,ij)\colon z_0+z_1j\mapsto z_1+z_0j$. It is an
involution that interchanges $V$ and $W$.
\item $\sigma_-\colon z_0+z_1j\mapsto z_1+\overline{z_0}j$. It is an
orientation-reversing isometry of order~$4$ that interchanges $V$ and $W$.
\end{enumerate}

The following relations among these isometries are easily checked:
\begin{enumerate}
\item $\sigma_-^2=\tau$.
\item $\sigma_+\tau=\tau\sigma_+$, and $\rho\tau=\tau\rho$.
\item $(\rho\sigma_+)^2=\tau$, so $\rho$ and $\sigma_+$ generate a
dihedral group of order~$8$.
\item $\sigma_+\sigma_-\sigma_+=\sigma_-^{-1}$, so $\sigma_+$ and 
$\sigma_-$ generate a dihedral group of order~$8$.
\end{enumerate}

\section{The Smale Conjecture}
\label{sec:smaleconj}

The original Smale Conjecture, proven by A. Hatcher \cite{HatcherSmale},
asserts that the inclusion $\Isom(S^3)\rightarrow \Diff(S^3)$ from the
isometry group to the diffeomorphism group is a homotopy equivalence. The
\emph{Generalized Smale Conjecture} (henceforth just called the Smale
Conjecture) asserts this for elliptic $3$-manifolds.

N. Ivanov~\cite{I5,I2} proved the Smale Conjecture for most of the
elliptic $3$-manifolds that contain one-sided Klein bottles, specifically:
\begin{enumerate}
\item[(i)] The lens spaces $L(4n,2n-1)$, $n\geq 2$
\item[(ii)] The quaternionic and prism manifolds for which $\pi_1(E)$ has a
nontrivial cyclic direct factor.
\end{enumerate}
The preprint~\cite{HKMR} gives proofs of the Smale Conjecture for all lens
spaces $L(m,q)$ with $m>2$, and for all quaternionic and prism manifolds.
Although the Smale Conjecture seems likely to hold for all elliptic
$3$-manifolds, no claim is currently asserted for the remaining
cases. Perelman's methods do not seem to apply, at least in their current
form (see~\cite[Section 1.4]{HKMR}).

\section{Heegaard splittings of elliptic $3$-manifolds: the genus-$0$ case}
\label{sec:genus0}

In this section we will prove Theorem~\ref{thm:S3} part (1), that is, that
$\H(S^3,S^2)\simeq \RP^3$.

Recall that $P\subset S^3$ is the $2$-sphere orthogonal to~$1$. The stabilizer
$\Isom_+(S^3,P)$ of $P$ in $\Isom_+(S^3)$ is exactly the stabilizer of the pair
$\{\pm1\}$, which is the subgroup $\O(3)\subset \SO(4)$. 

\begin{lemma}\label{lem:Pstabilizer}
The inclusion $\Isom(S^3,P)\to \Diff(S^3,P)$ is a homotopy equivalence.
\end{lemma}
\begin{proof}
Consider the diagram
\begin{equation*}
\begin{CD}
\Isom_+(S^3\rel P) @>>>  \Isom_+(S^3,P) @>>> \Isom(P)\\
@VVV @VVV @VVV\\
\Diff_+(S^3\rel P) @>>>  \Diff_+(S^3,P) @>>> \Diff(P)\\
\end{CD}
\end{equation*}
in which the vertical maps are inclusions, and the rows are fibrations, the
top row since it is a homomorphism of compact Lie groups, and the bottom
row by Lemma~\ref{lem:restrictToSigma}. The right vertical arrow is a
homotopy equivalence, by a theorem of S. Smale~\cite{Smale}. The left
vertical arrow is a homotopy equivalence, since the Smale Conjecture for
$S^3$ implies that $\Diff(D^3\rel \partial D^3)$ is contractible. Therefore
$\Isom_+(S^3,P)\to \Diff_+(S^3,P)$ is a homotopy equivalence. Since both
$\Isom(S^3,P)$ and $\Diff(S^3,P)$ contain orientation-reversing elements,
it follows that $\Isom(S^3,P)\to \Diff(S^3,P)$ is a homotopy equivalence.
\end{proof}

\begin{genus0theorem} $\H(S^3,S^2)\simeq \RP^3$.
\end{genus0theorem}

\begin{proof}
Since all $2$-spheres (smoothly) imbedded in $S^3$ are isotopic, we may
take $\Sigma$ to be $P$. Consider the diagram
\begin{equation*}
\begin{CD}
\Isom(S^3, P) @>>>  \Isom(S^3) @>>> \Isom(S^3)/\Isom(S^3,P)\\
@VVV @VVV @VVV\\
\Diff(S^3, P) @>>>  \Diff(S^3) @>>> \Diff(S^3)/\Diff(S^3,P)\\
\end{CD}
\end{equation*}
in which the vertical maps are inclusions. The rows are fibrations, the top
row since it is a homomorphism of compact Lie groups, and the bottom row by
Corollary~\ref{coro:basic_fibrations}. We have just seen that the left
vertical arrow is a homotopy equivalence. The middle vertical arrow is
the original Smale Conjecture, so we have
\begin{gather*}
\H(S^3,P)\simeq \Isom(S^3)/\Isom(S^3,P)\\
\simeq \Isom_+(S^3)/\Isom_+(S^3,P)
=\SO(4)/\O(3) = \RP^3\ ,
\end{gather*}
the latter equality since $\O(3)$ is the stabilizer of $\{\pm 1\}$ under
the transitive action of $\SO(4)$ on pairs of antipodal points in $S^3$.
\end{proof}

\section{Lens spaces}
\label{sec:goeritz_lens}

Our work on genus-$1$ splittings will require some information about lens
spaces, which we recall in this section. Let $L$ be the lens space
$L(m,q)$, with $m\geq 2$ and $q$ selected so that $1\leq q\leq m/2$. We
regard $L$ as $S^3/G_L$, where $G_L\subset S^1\ttimes S^1\subset \SO(4)$ is
the cyclic subgroup of order $m$ generated by
$\gamma_{m,q}=F(\xi_{2m}^{q+1},\xi_{2m}^{q-1})$.

For each $n\geq 1$, $L$ has a Heegaard surface $\Sigma_n$ of genus~$n$, and
by a theorem of F. Bonahon~\cite{Bonahon} for $n=1$ and Bonahon and
J.-P.~Otal~\cite{Bonahon-Otal} for $n\geq 2$, it is the unique Heegaard
surface of this genus up to isotopy. Consequently, $\H(L,\Sigma_n)$ is
path-connected.

The standard torus $\{z_0+z_1j\;|\;|z_0|=|z_1|\}\subset S^3$ is invariant
under the action of $G_L$, and its image under $S^3\to S^3/G_L=L$ is a
Heegaard torus in $L$. We denote the image by~$T$, and the solid tori in
$L$ bounded by $T$ by $V$ and $W$. 

In~\cite{Bonahon}, F. Bonahon proved that every diffeomorphism of $L$
preserves $T$ up to isotopy, and used this to calculate the mapping class
groups of lens spaces. To state the results, we first recall the isometries
of $L$ used in~\cite{Bonahon}, which we define here using the isometries
$\tau$, $\sigma_+$ and $\sigma_-$ of $S^3$ that were introduced near the
end of Section~\ref{sec:isometries} above.
\begin{enumerate}
\item For all $(m,q)$, $\tau\gamma_{m,q}\tau=\gamma_{m,q}^{-1}$, so $\tau$
induces an orientation-preserving involution of $L$, also denoted by
$\tau$, that restricts to the hyperelliptic involution on each of $V$
and~$W$.
\item When $q^2=1\bmod m$, we have 
$\sigma_+\gamma_{m,q}\sigma_+=\gamma_{m,q}^q$, so $\sigma_+$ induces an
orientation-preserving involution of $L$, also denoted by $\sigma_+$,
that interchanges $V$ and~$W$.
\item When $q^2=-1\bmod m$, we have
$\sigma_-\gamma_{m,q}\sigma_-^{-1}=\gamma_{m,q}^{-q}$, so $\sigma_-$
induces an orientation-reversing isometry of $L$, also denoted by
$\sigma_-$, that interchanges $V$ and~$W$.
\end{enumerate}
\begin{theorem}[F.\ Bonahon]\label{thm:ModL}
The groups $\Mod(L)$ are as follows:
\begin{enumerate}
\item $\Mod(L(2,1))=C_2$ generated by $\sigma_-$.
\item If $m>2$ and $q= 1$, then $\Mod(L)=C_2$ generated
by $\tau$.
\item If $m>2$ and $q^2=1\bmod m$ but $q\neq 1$, then
$\Mod(L)=C_2\times C_2$ generated by $\tau$ and $\sigma_+$.
\item If $m>2$ and $q^2=-1\bmod m$, then $\Mod(L)=C_4$ generated by
$\sigma_-$.
\item If $m>2$ and $q^2\neq \pm 1\bmod m$, then $\Mod(L)=C_2$ generated by
$\tau$.
\end{enumerate}
\end{theorem}
\noindent Note that Theorem~\ref{thm:ModL} implies the well-known fact that
$L(m,q)$ admits an orientation-reversing diffeomorphism if and only if
$q^2\equiv -1\bmod m$. Since each of the elements $\tau$, $\sigma_+$, and
$\sigma_-$ preserves $T$, Theorem~\ref{thm:ModL} also implies
\begin{corollary}\label{coro:H0LT}
$\Mod(L,T)\to \Mod(L)$ is surjective. 
\end{corollary}

It is not difficult to compute $G(L,T)$, and then $\Mod(L,T)$ using the exact 
sequence $1\to G(L,T)\to \Mod(L,T)\to \Mod(L)\to 1$. Since the proofs are
not difficult and we will not need the results, we simply record them here:
\begin{proposition}\label{prop:lensGoeritz1} 
\begin{enumerate}
\item If $q\neq 1$, then $G(L,T)=\{1\}$.
\item If $m>2$ and $q=1$, then $G(L,T)=C_2$, generated by $\sigma_+$, and
$G_0(L,T)=\{1\}$. 
\item $G(L(2,1),T)=C_2\times C_2$, generated by $\sigma_+$ and $\tau$,
and $G_0(L(2,1),T)=C_2$, generated by $\tau$.
\end{enumerate}
\end{proposition}
\noindent
\begin{proposition}\label{prop:lensGoeritz2}
\begin{enumerate}
\item If $q\neq 1$, then $\Mod(L,T)\to\Mod(L)$ is an isomorphism.
\item If $m>2$ and $q=1$, then $\Mod(L,T)=C_2\times C_2$, generated by
$\sigma_+$ and~$\tau$.
\item $\Mod(L(2,1),T)=D_8$, the dihedral group of order~$8$ generated by
$\sigma_+$ and $\sigma_-$.
\end{enumerate}
\end{proposition}

\section{Heegaard splittings of elliptic $3$-manifolds: the genus-$1$ cases}
\label{sec:lens}

In this section we will prove Theorem~\ref{thm:S3}(2) and the (a)
statements in all three cases of Theorem~\ref{thm:lens_spaces}. We will
retain the notation of Section~\ref{sec:goeritz_lens}, so that $L$ is the
lens space $L(m,q)$ with $m\geq 2$ and $1\leq q\leq m/2$, except that we
now allow $L=L(1,0)$, the $3$-sphere. As in Section~\ref{sec:goeritz_lens},
$L$ is regarded as $S^3/G_L$, where $G_L\subset S^1\ttimes S^1\subset
\SO(4)$ is the subgroup generated by
$\gamma_{m,q}=F(\xi_{2m}^{q+1},\xi_{2m}^{q-1})$. In particular,
$\gamma_{1,0}=F(-1,-1)=1_{\SO(4)}$, and $\gamma_{2,1}=F(-1,1)=\alpha$, the
antipodal map.

Recall that $\Isom(L)=\Norm(G_L)/G_L$, where $\Norm(G_L)$ is the normalizer
of $G_L$ in $\O(4)$. Consequently, an element $F(z,w)$ in $\Norm(G_L)\cap
\SO(4)$ induces an isometry on $L$, which we denote by~$f(z,w)$.

We will need to know the groups $\Norm(G_L)\cap \SO(4)$, which we denote by
$\Norm_+(G_L)$. In the following lemma, $\Dih(S^1\ttimes S^1)$ denotes the
subgroup of index $2$ in $\Ostar\ttimes \Ostar$ generated by $S^1\ttimes
S^1$ and the involution $\tau=F(j,j)$, which acts by inversion on elements
of $S^1\ttimes S^1$. From~\cite{McC} we have the following information.
\begin{lemma}\label{lem:GLnormalizers}
\begin{enumerate}
\item[(i)] For $m\leq 2$, $\Norm_+(G_L)=\SO(4)$.
\item[(ii)] For $m>2$ and $q=1$, $\Norm_+(G_L) =S^3\ttimes \Ostar$.
\item[(iii)] For $m>2$, $q>1$ and $q^2\equiv 1\bmod m$,
$\Norm_+(G_L) =\Ostar\ttimes \Ostar$.
\item[(iv)] For $m>2$ and $q^2\not\equiv 1\bmod m$,
$\Norm_+(G_L) =\Dih(S^1\ttimes S^1)$.
\end{enumerate}
\end{lemma}
\begin{proof}
Part (i) is obvious. Part (ii) is found in Case~III on p.~175
of~\cite{McC}, and
Part~(iii) is found in Case~VI on p.~176 of~\cite{McC}. Part~(iv) is found
in Cases~IV and~V on p.~175 of~\cite{McC}.
\end{proof}

As in Section~\ref{sec:goeritz_lens}, $T$ is the standard Heegaard torus
for $L$. In particular, for $L=L(1,0)$, $T$ is
$\{z_0+z_1j\;|\;|z_0|=|z_1|\}$. We will need to know the groups
$\Isom_+(L,T)$.
\begin{lemma}\label{lem:stabT}
\begin{enumerate}
\item $\isom(L,T)=(S^1\ttimes S^1)/G_L$
\item When $q^2\equiv 1\bmod m$,
$\Isom_+(L,T)=(\Ostar\ttimes \Ostar)/G_L$.
\item When $q^2\not \equiv 1\bmod m$,
$\Isom_+(L,T)=\Dih(S^1\ttimes S^1)/G_L$.
\end{enumerate}
\end{lemma}
\begin{proof}
Suppose first that $L=L(1,0)=S^3$.  It is straightforward to check that
$\Ostar\ttimes \Ostar \subset \Isom_+(S^3,T)$. Suppose that $F(z,w)\in
\Isom_+(S^3,T)$. Now $T$ is exactly the set of points equidistant from the
two geodesics $S^1$ and $S^1j$, in fact these are exactly the most distant
points from it.  Since $F(z,w)$ preserves $T$, it must preserve $S^1\cup
S^1j$. A quick check shows that~$z,w\in \Ostar$ (starting with the case
when $F(z,w)(1),F(z,w)(i)\in S^1$ and $F(z,w)(j),F(z,w)(ij)\in S^1j$, we
compute that either $(z,w)\in S^1\times S^1$ or $(z,w)\in S^1j\times S^1j$,
while when $F(z,w)(S^1)=S^1j$, the previous case applies to $F(z,w)F(1,j)$
showing that $(z,w)\in S^1\times S^1j$ or $(z,w)\in S^1j\times S^1$).

In general, we have
\begin{gather*}
\Isom_+(L,T) = (\Isom_+(S^3,T)\cap \Norm(G_L))/G_L\\
=(\Ostar\ttimes \Ostar\cap \Norm(G_L))/G_L\ .
\end{gather*}
From Lemma~\ref{lem:GLnormalizers}, $\Ostar\ttimes \Ostar\cap \Norm(G_L)$ is
$\Ostar\ttimes \Ostar$ when $q^2\equiv 1\bmod m$ and is $\Dih(S^1\ttimes
S^1)$ when $q^2\not\equiv 1\bmod m$. This establishes statements~(2)
and~(3). The description of $\isom(L,T)$ in (1)~follows directly.
\end{proof}

\begin{lemma}\label{lem:Tstabilizer}
The inclusion $\Isom(L,T)\to \Diff(L,T)$ is a homotopy equivalence.
\end{lemma}
\begin{proof}
From Theorem~\ref{thm:ModL}, $L$ admits an orientation-reversing
diffeomorphism only when $q^2\equiv -1\bmod m$, in which case $\sigma_-$ is
an orientation-reversing element of $\Isom(L,T)$. That is, $\Diff(L,T)$
contains orientation-reversing elements if and only if $\Isom(L,T)$ does.
Therefore it suffices to prove that the inclusion $k\colon \Isom_+(L,T)\to
\Diff_+(L,T)$ is a homotopy equivalence.

We first check that $k$ is injective on path components. By
Lemma~\ref{lem:stabT}, $\Isom_+(L,T)$ has either two or four components,
represented by $1_L$, $\tau$, and when there are four, $\sigma_+=f(i,ij)$
and $\sigma_+\tau=f(ij,-i)$. Of these, only the elements of $\isom(L,T)$
preserve the sides and are isotopic to the identity on $T$, so
$\Isom_+(L,T)\cap \diff(L,T)=\isom(L,T)$.

To see that $k$ is surjective on path components, let $f\in
\Diff_+(L,T)$. If $f$ interchanges the sides of $T$, then a well-known
homology argument shows that $q^2\equiv 1\bmod m$. (Let $\mu$ and $\lambda$
in $H_1(T)$ be a meridian and longitude for $T\subset V$ such that
$m\lambda + q\mu$ is a meridian of $T\subset W$, and write $h=f|_T$. For
$h_*\colon H_1(T)\to H_1(T)$, $h_*(\mu)$ is a meridian for $T\subset W$, so
$h_*(\mu)=\epsilon(m\lambda + q\mu)$ for $\epsilon$ either $1$ or
$-1$. Writing $h_*(\lambda)=\epsilon(a\lambda+b\mu)$,
$\det(h_*)=aq-mb$. Since $h$ interchanges the sides of $T$,
$\pm\mu=mh_*(\lambda)+qh_*(\mu)$, implying that $a=-q$ and hence
$\det(h_*)=-q^2-mb$. When $f$ is orientation-preserving, $h$ must be
orientation-reversing, giving $q^2\equiv 1\bmod m$.) Since $q^2\equiv
1\bmod m$, $\sigma_+$ is an element of $\Isom_+(L,T)$, and composing it
with $f$, we may assume that $f$ preserves the sides of $T$. Since $f$ must
then preserve the meridian curves of both complementary tori up to isotopy,
it is isotopic either to the identity on both solid tori or to the
hyperelliptic involution on both. In the latter case, we may compose $f$
with $\tau$, an element of $\Isom_+(L,T)$, to assume that $f$ is isotopic
to the identity on both sides and therefore lies in $\diff(L,T)$. Therefore
every path component of $\Diff(L,T)$ contains elements of~$\Isom(L,T)$.

From Lemma~\ref{lem:stabT}, $\isom(S^3,T)$ is a full $S^1\times S^1$
subgroup of isometries in $\diff(L,T)$, so $k$ is a homotopy equivalence on
each path component.
\end{proof}
\longpage 

\begin{genus1theorem}\label{thm:genus1_cases}
Let $\Sigma_1$ be the standard genus-$1$ Heegaard surface
in $S^3$, $\RP^3$, or a lens space $L(m,q)$ with $1\leq q<m/2$. Assume, if
necessary, that the $3$-manifold satisfies the Smale Conjecture.
\begin{enumerate}
\item[(1)] For $L=S^3$ or $\RP^3$, $\H(L,\Sigma_1)\simeq \RP^2\times \RP^2$.
\item[(2)] For $m\geq 3$, $\H(L(m,1),\Sigma_1)\simeq \RP^2$.
\item[(3)] For $q\geq 2$, $\H(L(m,q),\Sigma_1)$ is contractible.
\end{enumerate}
\end{genus1theorem}

\begin{proof} Let $L$ be any of these manifolds. By a theorem of 
F. Waldhausen~\cite{Waldhausen} for $S^3$ and F. Bonahon~\cite{Bonahon} for
$\RP^3$ and $L(m,q)$, $T$ is the unique Heegaard torus of $L$ up to
isotopy. So we may take $\Sigma_1=T$.

Consider the commutative diagram
\begin{equation*}
\begin{CD}
\Isom(L,T) @>>>  \Isom(L) @>>> \Isom(L)/\Isom(L,T)\\
@VVV @VVV @VVV\\
\Diff(L,T) @>>> \Diff(L) @>>> \Diff(L)/\Diff(L,T)\ ,
\end{CD}
\end{equation*}
whose vertical maps are inclusions. The rows are fibrations, the
first since $\Isom(L,T)$ is a closed subgroup of the Lie group $\Isom(L)$,
and the second by Corollary~\ref{coro:basic_fibrations}.  Since $L$ is
assumed to satisfy the Smale Conjecture, the middle arrow is a homotopy
equivalence. By Lemma~\ref{lem:Tstabilizer}, the first vertical arrow is a
homotopy equivalence, so we have $\H(L,T)\simeq \Isom(L)/\Isom(L,T)$. Since
$L$ has an orientation-reversing isometry only when when $q^2\equiv -1\bmod
m$, in which case $\sigma_-\in \Isom(L,T)$ is an orientation-reversing
isometry, we have
\begin{gather*}
\Isom(L)/\Isom(L,T)=\Isom_+(L)/\Isom_+(L,T)\\
=(\Norm_+(G_L)/G_L)/(\Ostar\ttimes \Ostar\cap \Norm(G_L))/G_L\\
=\Norm_+(G_L)/(\Ostar\ttimes \Ostar\cap \Norm(G_L))\ .
\end{gather*}

Using Lemma~\ref{lem:GLnormalizers}, we can now calculate
$\H(L,T)$. For $L=S^3$ and $L=\RP^3$,
\begin{gather*}\H(L,T) \simeq \SO(4)/(\Ostar\ttimes \Ostar)
=(S^3\ttimes S^3)/(\Ostar\ttimes \Ostar)\\
=(S^3\times S^3)/(\Ostar\times \Ostar)
=S^3/\Ostar \times S^3/\Ostar
= \RP^2\times \RP^2\ .
\end{gather*}
For $L=L(m,1)$, 
\begin{gather*}\H(L,T) \simeq (S^3\ttimes \Ostar)/(\Ostar\ttimes \Ostar)\\
=(S^3\times \Ostar)/(\Ostar\times \Ostar)
=S^3/\Ostar =\RP^2\ .
\end{gather*}
For $L=L(m,q)$, $q>1$, 
\[\H(L,T) \simeq (\Ostar\ttimes \Ostar)/(\Ostar\ttimes \Ostar)\ ,\]
a single point.
\end{proof}

\section{Heegaard splittings of elliptic $3$-manifolds: genus
$2$ and higher}
\label{sec:elliptics}

We continue to use the notation of Section~\ref{sec:lens}.
\begin{genus2theorem}\label{thm:high-genus_cases}
For $n\geq 2$, let $\Sigma_n$ be the standard genus-$n$ Heegaard surface in
one of $S^3$, $\RP^3$, or a lens space $L(m,q)$ with $1\leq q<m/2$. Assume,
if necessary, that the $3$-manifold satisfies the Smale Conjecture.
\begin{enumerate}
\item[(1)] $\H_i(S^3,\Sigma_n)\cong \pi_i(S^3\times S^3)$ for
$i\geq 2$, and there is a non-split exact sequence
\[ 1\to C_2 \to H_1(S^3,\Sigma_n) \to G(S^3,\Sigma_n)\to 1\ .\]
\item[(2)] $\H_i(\RP^3,\Sigma_n)\cong \pi_i(S^3\times S^3)$ 
for $i\geq 2$, and there is an exact sequence
\[ 1\to C_2\times C_2\to \H_1(\RP^3,\Sigma_n) \to G(\RP^3,\Sigma_n)\to 1\ .\]
\item[(3)] For $m\geq 3$, $\H_i(L(m,1),\Sigma_n)\cong \pi_i(S^3)$ for
$i\geq 2$, and there are exact sequences
\[ 1\to \Z\to \H_1(L(m,1),\Sigma_n) \to G(L,\Sigma_n)\to 1\]
for $m$ odd, and
\[ 1\to \Z\times C_2\to \H_1(L(m,1),\Sigma_n) \to G(L,\Sigma_n)\to 1\]
for $m$ even.
\item[(4)] For $q\geq 2$, $\H_i(L(m,q),\Sigma_n)=0$ for $i\geq 2$, 
and there is an exact sequence
\[ 1\to \Z\times \Z\to \H_1(L,\Sigma_n) \to G(L,\Sigma_n)\to 1\ .\]
\end{enumerate}
\end{genus2theorem}

\begin{proof} Let $L$ denote any of these manifolds. By results of
Waldhausen~\cite{Waldhausen} for $S^3$ and Bonahon-Otal~\cite{Bonahon-Otal}
for $\RP^3$ and $L(m,q)$, $1\leq q<m/2$, $L$ has a unique Heegaard surface
$\Sigma_n$ for each genus greater than $1$, so the corresponding spaces of
Heegaard splittings are path-connected.

By Theorem~\ref{thm:genus2case}, $\H_i(L,\Sigma_n)\cong
\pi_i(\Diff(L))$, and there is an exact sequence
\[1\to \pi_1(\Diff(L))\to \H_1(L,\Sigma_n)
\to G(L,\Sigma_n) \to 1\ .\] 
Since we are assuing that $L$ satisfies the Smale Conjecture, the groups
$\pi_i(\Diff(L))\cong \pi_i(\Isom(L))$ for $i\geq 1$ can be found using
Table~\ref{tab:lens}.\par\smallskip

For $L=S^3$, we have $\isom(S^3)=\SO(4)$, so
$\H_i(S^3,\Sigma_n)=\pi_i(\SO(4))\cong \pi_i(S^3\times S^3)$ for $i\geq 2$
and $\pi_1(\Diff(S^3))$ is $C_2$. To see that the exact sequence in
part~(a) does not split, observe that there is an isotopy $J_t$ with
$J_0=1_{S^3}$ and $J_1$ a hyperelliptic involution on $\Sigma_n$; $J_t$
rotates through an angle $\pi$ around an axis of symmetry of
$\Sigma_n$. This defines an element of $\H_1(S^3,\Sigma_n)$ whose square is
the generator of $\pi_1(\Diff(S^3))$. A normal $C_2$-subgroup is central,
so if the exact sequence split we would have
$\H_1(S^3,\Sigma_n)=\pi_1(\Diff(S^3))\times G(S^3,\Sigma_n)$, and the
generator of $\pi_1(\Diff(S^3))$ could not be a square.

For  $L=\RP^3$,
$\isom(\RP^3)$ is homeomorphic to $\SO(3)\times \SO(3)$,
so $\H_i(\RP^3,\Sigma_n)=\pi_i(S^3\times S^3)$ for $i\geq
2$ and $\pi_1(\Diff(\RP^3))$ is $C_2\times C_2$. 

For $m\geq 3$, $\isom(L(m,1))$ is homeomorphic to $S^1\times S^3$
or to $S^1\times \SO(3)$ according as $m$ is odd or even. So
$\H_i(L(m,1),\Sigma_n)=\pi_i(S^3)$ for $i\geq
2$, while $\pi_1(\Diff(L(m,1)))$ is $\Z$ or $\Z\times C_2$ according as $m$
is odd or even.

For $q\geq 2$, $\isom(L(m,q))$ is homeomorphic to $S^1\times S^1$, so
$\H_i(L(m,q),\Sigma_n)=0$ for $i\geq 2$ and $\pi_1(\Diff(L(m,q)))$ is $\Z\times
\Z$.
\end{proof}

\begin{threefibertheorem} Let $E$ be an elliptic $3$-manifold, but not $S^3$ or a lens
space. Assume, if necessary, that $E$ satisfies the Smale Conjecture.  Let
$\Sigma$ be a Heegaard surface in $E$.
\begin{enumerate}
\item If $\pi_1(E)\cong D^*_{4m}$, or if $E$ is one of the three manifolds
with fundamental group either $T^*_{24}$, $O^*_{48}$, or $I^*_{120}$, then
$\H_i(E,\Sigma)\cong\pi_i(S^3)$ for $i\geq 2$ and there is an exact
sequence
\[ 1\to C_2\to \H_1(E,\Sigma) \to G(E,\Sigma)\to 1\ .\]
\item If $E$ is not one of the manifolds in Case~(1), that
is, either $\pi_1(E)$ has a nontrivial cyclic direct factor, or $\pi_1(E)$
is a diagonal subgroup of index~$2$ in $D^*_{4m}\times C_n$ or of index~$3$
in $T^*_{48}\times C_n$, then $\H_i(E,\Sigma)=0$ for $i\geq 2$, and
there is an exact sequence
\[ 1\to \Z\to \H_1(E,\Sigma) \to G(E,\Sigma)\to 1\ .\]
\end{enumerate}
\end{threefibertheorem}

\begin{proof}
Fix a Heegaard surface
$\Sigma$ in the elliptic $3$-manifold $E$. Since
$E$ is not $S^3$ or a lens space, $\Sigma$ has genus at least $2$.
By Theorem~\ref{thm:genus2case}, $\H_i(E,\Sigma)\cong \pi_i(\Diff(E))$ for
$i\geq 2$, and there is an exact sequence
\[1\to \pi_1(\Diff(E))\to \H_1(E,\Sigma)
\to G(E,\Sigma) \to 1\ .\] 
Since $E$ is assumed to satisfy the Smale Conjecture,
$\pi_i(\Diff(M))\cong \pi_i(\Isom(E))$.\par\smallskip

\noindent\textsl{Case I:} $\pi_1(E)\cong D^*_{4m}$, or $E$ is one of the three
manifolds with fundamental group either $T^*_{24}$, $O^*_{48}$, or
$I^*_{120}$.
\smallskip

Referring to Table~\ref{tab:nonlens}, we see that $\isom(E)$ is
homeomorphic to $\SO(3)$, so $\H_i(E,\Sigma)=\pi_i(S^3)$ for $i\geq 2$ and
$\pi_1(\isom(E))\cong C_2$.
\smallskip

\noindent\textsl{Case II:} $E$ is not one of the manifolds in Case~I, that
is, either $\pi_1(E)$ has a nontrivial cyclic direct factor, or $\pi_1(E)$
is a diagonal subgroup of index~$2$ in $D^*_{4m}\times C_n$ or of index~$3$
in $T^*_{48}\times C_n$.
\smallskip

Again from Table~\ref{tab:nonlens}, $\isom(E)$ is homeomorphic to $S^1$, so
$\H_i(E,\Sigma)=0$ for $i\geq 2$ and $\pi_1(\isom(E))$ is $\Z$.
\end{proof}

For the manifolds in Theorem~\ref{thm:elliptics}, M. Boileau and
J.-P.~Otal~\cite{Boileau-Otal} have proven that there is a unique genus-$2$
Heegaard splitting up to isotopy, so in that case $\H(L,\Sigma)$ is known
to be connected.

\bibliographystyle{amsplain}

\end{document}